\documentclass{amsart}
\usepackage{amsmath, amssymb}
\usepackage{color,mathdots}
\usepackage{url}

\newtheorem{theorem}{Theorem}[section]
\newtheorem{lemma}[theorem]{Lemma}
\newtheorem{corollary}[theorem]{Corollary}
\newtheorem{fact}[theorem]{Fact}
\newtheorem{proposition}[theorem]{Proposition}

\theoremstyle{definition}
\newtheorem{example}[theorem]{Example}
\newtheorem{question}[theorem]{Question}
\newtheorem{remark}[theorem]{Remark}

\newtheorem{assumption}[theorem]{Assumption}

\def \U {\mathcal U}

\def \C {\mathcal C}

\def\DCF{\operatorname{DCF}}
\def\ccm{\operatorname{CCM}}

\def\dcl{\operatorname{dcl}}
\def\acl{\operatorname{acl}}

\def\\tp{\operatorname{\tp}}

\def\tp{\operatorname{tp}}
\def\stp{\operatorname{stp}}
\def\cb{\operatorname{Cb}}

\def\RM{\operatorname{RM}}

\def\Ind#1#2{#1\setbox0=\hbox{$#1x$}\kern\wd0\hbox to 0pt{\hss$#1\mid$\hss}
\lower.9\ht0\hbox to 0pt{\hss$#1\smile$\hss}\kern\wd0}
\def\ind{\mathop{\mathpalette\Ind{}}}
\def\Notind#1#2{#1\setbox0=\hbox{$#1x$}\kern\wd0\hbox to 0pt{\mathchardef
\nn=12854\hss$#1\nn$\kern1.4\wd0\hss}\hbox to
0pt{\hss$#1\mid$\hss}\lower.9\ht0 \hbox to
0pt{\hss$#1\smile$\hss}\kern\wd0}
\def\nind{\mathop{\mathpalette\Notind{}}}

\begin{document}

\title[Isolated types of finite rank]{Isolated types of finite rank:\\ an abstract Dixmier-Moeglin equivalence}

\author{Omar Le\'on S\'anchez}
\address{School of Mathematics, University of Manchester, Oxford Road, Manchester, United Kingdom M13 9PL}
\email{omar.sanchez@manchester.ac.uk}

\author{Rahim Moosa}
\address{Department of Pure Mathematics, University of Waterloo, 200 University Avenue West, Waterloo, Ontario, Canada N2L 3G1}
\email{rmoosa@uwaterloo.ca}

\date{\today}
\thanks{{\em Acknowledgements}: R. Moosa was partially supported by an NSERC Discovery Grant.}
\subjclass[2010]{03C95, 03C98, 12H05}
\keywords{model theory, totally transcendental theories, differential fields}

\begin{abstract}
Suppose $T$ is totally transcendental and every minimal non-locally-modular type is nonorthogonal to a nonisolated minimal type over the empty set.
It is shown that a finite rank type $p=\tp(a/A)$ is isolated if and only if $\displaystyle a\ind_{Ab}q(\U)$ for every $b\in \acl(Aa)$ and $q\in S(Ab)$ nonisolated and minimal.
This applies to the theory of differentially closed fields -- where it is motivated by the differential Dixmier-Moeglin equivalence problem -- and the theory of compact complex manifolds.
\end{abstract}

\maketitle


\section{Introduction}

\noindent
Let $T$ be a complete totally transcendental theory admitting elimination of imaginaries, and $\U\models T$ a sufficiently saturated model.
We are interested in the following condition on a type $p\in S(A)$.
\begin{itemize}
\item[($\dagger$)]
Suppose $a\models p$, $b\in \acl(Aa)$, and $q\in S(Ab)$ is nonisolated and minimal. 
Then $\displaystyle a\ind_{Ab}q(\U)$.
\end{itemize}
By a {\em minimal} type we mean one that is stationary and of $U$-rank one.
The condition~($\dagger$) is essentially about the relationship (or rather lack thereof) between~$p$ and the nonisolated minimal complete types of the theory -- though the specific choice of parameters involved here are important.
We will show that in certain theories of interest (including differentially closed fields and compact complex manifolds), and assuming that $p$ is of finite rank, this condition characterises when $p$ is isolated.
It should be viewed as a reduction of the study of isolation from the finite rank case to the minimal case.

Another motivation for~($\dagger$) comes from an application of the model theory of  differentially closed fields of characteristic zero ($\DCF_0$) to a problem in noncommutative algebra.
It was observed in~\cite{pdme} that the classical Dixmier-Moeglin equivalence for noetherian algebras is connected to the relationship in $\DCF_0$ between a finite rank type over constant parameters being isolated and being weakly orthogonal to the field of constants.
The fact that these are not equivalent lead in~\cite{pdme} to the first counterexample to the Poisson Dixmier-Moeglin equivalence, and the first finite Gelfand-Kirillov dimension counterexample to the classical Dixmier-Moeglin equivalence.
Now, $p\in S(A)$ being weakly orthogonal to the constants is precisely the instance of~($\dagger$) when $b=\emptyset$ and $q$ is the generic type of the constants.
The counterexample in~\cite{pdme} was nonisolated and satisfied this instance of~($\dagger$) but failed another instance; one where $q$ was the generic type of a Manin kernel of a simple abelian variety not descending to the constants.
So, the equivalence of~($\dagger$) and isolation for finite rank types can be viewed as an abstract resolution to the Dixmier-Moeglin equivalence problem, where in order to get a true statement we have to replace weak orthogonality to the constants by all instances of~($\dagger$).

Our main result here is the following:
{\em Suppose that in $T$ every minimal non-locally-modular type is nonorthogonal to a nonisolated minimal type over the empty set.
Then a finite rank $p\in S(A)$ is isolated if and only if it satisfies}~($\dagger$).
This is Theorem~\ref{the1} below.
Specialising to the case of $T=\DCF_0$, we obtain  in Theorem~\ref{the1dcf} a more concrete form which we point out yields a quick proof of the differential Dixmier-Moeglin equivalence for certain $D$-varieties that were considered in~\cite{BLSM}.
Also in~$\S$\ref{sectdcf} we give examples showing that our characterisation of isolation cannot be substantially improved, in that we really have to range over all $b\in\acl(Aa)$ in the formulation of~($\dagger$).

We will use without extensive explanation various notions and facts from geometric stability theory -- we suggest~\cite{gst} as a general reference.
The underlying total transcendentality assumption is so that prime models over sets exist and are unique up to isomorphism, and $p\in S(A)$ is isolated if and only if it is realised in a prime model over $A$.
In particular, we freely use the following properties:
\begin{enumerate}
\item
$\tp(a/Ab)$ and $\tp(b/A)$ are isolated if and only if $\tp(ab/A)$ is isolated.
\item $p=\tp(a/A)$ is isolated if and only if $\stp(a/A)$ is isolated.
\item
If $p$ is nonisolated then so is any nonforking extension.
\end{enumerate}
Points~(1) and~(2) follow easily using prime models.
For point~(3), note that if $q=\tp(a/B)$ is isolated by $\phi(x,b)$, and $r(y):=\tp(b/A)$ where $A=\acl(A)\subseteq B$ is such that $q$ does not fork over $A$, then the $\phi(x,y)$-definition of $r(y)$ isolates $\tp(a/A)$.
We also use the definable binding group theorem in totally transcendental theories.

\bigskip
\section{Necessity of~($\dagger$)}
\label{sectnec}

\noindent
Without any assumptions beyond total transcendentality, we can show that~($\dagger$) is necessary:

\begin{proposition}\label{isotodagger}
If $p\in S(A)$ is isolated then~{\em ($\dagger$)} holds.
\end{proposition}

\begin{proof}
Note that $\displaystyle a\ind_{C}q(\U)$ where $C=\dcl(Aba)\cap\dcl(Abq(\U))$.
Indeed, this follows from stable embeddedness, see for example Lemma~1 in the Appendix of~\cite{zoe-udi}.
In fact, $\tp(a/C)\vdash\tp(a/Abq(\U))$.
In any case, it suffices to show that $C\subseteq\acl(Ab)$.

Let $M$ be a prime model over $Ab$.
Since $\tp(a/A)$ and $\tp(b/Aa)$ are isolated, so is $\tp(a/Ab)$.
By automorphisms, we may assume that $a$ is in~$M$.
It follows that $C\subseteq\dcl(Aba)\subseteq M$.

Now, let $c\in C$.
Since $c\in\dcl(Abq(\U))$ we can write $c=f(e)$ where $f$ is $Ab$-definable and $e$ is a finite tuple from $q(\U)$.
On the other hand, as $M$ is a model, we can find $e'$ from $M$ such that $c=f(e')$ as well.
We claim that $\displaystyle e'\ind_{Ab} e$.
This will suffice, as then $\dcl(Abe)\cap \dcl(Abe')\subseteq\acl(Ab)$, and hence $c=f(e)=f(e')$ is in $\acl(Ab)$, as desired.

As was pointed out to us by the anonymous referee, that $\displaystyle e'\ind_{Ab} e$ in turn follows from the following general, and probably well known, lemma.
\end{proof}

\begin{lemma}
Suppose $q\in S(B)$ is a minimal nonisolated type and $M$ is a prime model over $B$.
Then $\displaystyle M\ind_Bq(\U)$.
\end{lemma}

\begin{proof}
It suffices to show that if $\alpha$ is a finite sequence of realisations of~$q$, then the canonical base, $c:=\cb(\alpha/M)$, is contained in $\acl(B)$.
Suppose, toward a contradiction, that $c\notin\acl(B)$.

Note that $c\in M\cap \dcl\big(q(\U)\big)$.
Let $D$ be a finite set of realisations of $q$ such that $c\in\dcl(D)$.
Then $c\nind_BD$.

Let $\phi\in q$ be of minimal Morley rank and degree.
(Note that the Morley rank of $q$ need not be one.)
Then $c\in M\cap\dcl\big(\phi(\U)\big)=\dcl\big(\phi(M)\big)$.
Let $A\subseteq \phi(M)$ be a finite subset such that $c\in\dcl(A)$.
Then $A\nind_BD$.
Hence $a\nind_{B'}d$ for some $a\in A,d\in D$ and $B'\supseteq B$.
By minimality of $q$, it follows that $d\in\acl(B'a)\setminus\acl(B')$ and $\tp(d/B')$ is a nonforking extension of $q$.

By choice, $\phi$ isolates $q$ among those types over $B$ with Morley rank $\geq \RM(q)$.
And $\phi\in\tp(a/B)$.
But $\tp(a/B)\neq q$ since the former is isolated, being realised in $M$, and the latter is not isolated by assumption.
We must therefore have that $\RM(a/B)<\RM(q)$.
Hence $\RM(a/B')\leq\RM(a/B)<\RM(q)=\RM(d/B')$.
This contradicts $d\in\acl(B'a)$.
\end{proof}

\begin{remark}
The proof of the above proposition actually gives us something stronger: in~($\dagger$) we can range over all $b$ such that $\tp(b/Aa)$ is isolated, rather than asking for $b\in\acl(Aa)$.
\end{remark}

But what we are really interested in is not strong consequences of isolation, but rather, weak sufficient conditions.

\begin{question}
Does~($\dagger$) characterise isolation of a {\em finite rank} type $p\in S(A)$?
\end{question}

\bigskip
\section{Sufficiency of~($\dagger)$}
\label{sectchar}

\noindent
We make an additional assumption on $T$ under which~$(\dagger)$ becomes also a sufficient condition for isolation of finite rank types.

\begin{assumption}
\label{assumption}
Every complete non-locally-modular minimal type is nonorthogonal to a nonisolated minimal type in $S(\emptyset)$.
\end{assumption}

This is satisfied in $\DCF_0$ and $\ccm$ because of the particular manifestations of the Zilber dichotomy in these theories:
In $\ccm$ every non-locally-modular minimal type is nonorthogonal to the generic type of the projective line, and in $\DCF_0$ every non-locally-modular minimal type is nonorthogonal to the generic type of the field of constants.
These are minimal types over the empty set because the projective line in $\ccm$ and the field of constants in $\DCF_0$ are both $0$-definable strongly minimal sets.
They are nonisolated because the projective line and the field of constants both have infinitely many points in $\acl(\emptyset)$.

\begin{theorem}\label{the1}
Suppose $T$ is totally transcendental and satisfies Assumption~\ref{assumption}.
Let $p\in S(A)$ be of finite rank.
Then $p$ is isolated if and only if it satisfies~{\em($\dagger$)}.
\end{theorem}

\begin{proof}
That isolated types satisfy~($\dagger$) in arbitrary totally transcendental theories is the content of Proposition~\ref{isotodagger}.

Before proving the converse, we first observe that if $\tp(a/A)$ satisfies~($\dagger$) then for any $e\in\acl(Aa)$ so do  $\tp(e/A)$ and $\tp(a/Ae)$.
For the former, let $b\in\acl(Ae)$, and $q\in S(Ab)$ nonisolated and minimal.
Then as $b\in\acl(Aa)$ also, ($\dagger$) implies that $\displaystyle a\ind_{Ab} q(\U)$, and hence $\displaystyle e\ind_{Ab}q(\U)$.
To see that $\tp(a/Ae)$ satisfies~($\dagger$), let $b\in \acl(Aea)$ and $q\in S(Aeb)$ nonisolated and minimal.
Then as $eb\in\acl(Aa)$ and $\tp(a/A)$ satisfies $(\dagger)$, we get $\displaystyle a\ind_{Aeb}q(\U)$, as desired.

We proceed to prove by induction on $U(p)$ that if $p$ satisfies~($\dagger$) then it is isolated.
If $U(p)=0$ then it is isolated.
If $U(p)=1$ then isolation is immediate from~$(\dagger)$ applied with $b=\emptyset$.
So assume that $U(p)>1$.

Suppose $p=\tp(a/A)$ is orthogonal to all non-locally-modular minimal types.
Since $p$ is of finite rank it is nonorthogonal to some minimal type, which by assumption must be locally modular.
Now, the minimal case of an old result of Hrushovski's on nonorthogonality to locally modular regular types implies that there exists $e\in\acl(Aa)$ with $U(e/A)=1$, see~\cite[Theorem~2]{udi-locmod}.
Since $U(p)>1$, $a\notin\acl(Ae)$.
Hence both $U(e/A)$ and $U(a/Ae)$ are less than $U(p)$, and as we have observed, both $\tp(e/A)$ and $\tp(a/Ae)$ satisfy~($\dagger$).
By induction, they are both isolated.
Hence $\tp(a/A)$ is isolated, as desired.

It remains to consider the case when $p$ is nonorthogonal to some non-locally-modular minimal type.
By Assumption~\ref{assumption}, $p$ is nonorthogonal to a nonisolated minimal type in  $S(\emptyset)$.
Taking the nonforking extension of that type to $A$ we get $q\in S(A)$ minimal, nonisolated, and nonorthogonal to $p$.
Applying~$(\dagger)$ with $b=\emptyset$, we have that $a\ind_A q(\U)$.
On the other hand,  nonorthogonality implies that there exists  $d\in \dcl(Aa)\setminus A$ with $\tp(d/A)$ internal to $q$, see~\cite[Corollary~7.4.6]{gst}.
We know that both $\tp(d/A)$ and $\tp(a/Ad)$ satisfy~($\dagger$).
If $U(d/A)< U(p)$ then, by induction, both $\tp(d/A)$ and $\tp(a/Ad)$ are isolated; and consequently $p$ would be isolated.
So we may assume $U(d/A)=U(p)$.
That is,  $a$ and $d$ are interalgebraic over $A$.
Let $p'=\stp(d/A)$ and $G$ be the $q$-binding group of $p'$.
This is a definable group over $\acl(A)$ acting definably over $A$ on $p'(\U)$.
Moreover, $a\ind_A q(\U)$ implies $d\ind_A q(\U)$, so that the binding group acts transitively on $p'(\mathcal U)$.
But this implies that $p'(\mathcal U)=G\cdot d$ is a definable set.
As $p'(\U)$ is also invariant under $\operatorname{Aut}_{\acl(A)}(\U)$, it follows that $p'(\U)$ is an $\acl(A)$-definable set.
That is, $p'$ is isolated.
But $a\in\acl(Ad)$ now implies that $\stp(a/A)$, and hence $\tp(a/A)=p$, is isolated.
\end{proof}

\bigskip
\section{The case of $\DCF_0$}
\label{sectdcf}

\noindent
In this section we specialise to the case when $T$ is the theory of differentially closed fields in characteristic zero.
We will make use of the Zilber trichotomy in this theory, a deep and fundamental result of Hrushovski and Sokolovic describing what the nontrivial minimal types look like.
While this work appears only in the unpublished manuscript~\cite{hrushovski-sokolovic}, some details on the locally modular case can be found in~\cite[III.4]{mof}, while the non-locally-modular case has an alternative proof that appears in~\cite{pillayziegler}.
We suggest~\cite{markermanin} for an exposition on the Manin kernels associated to abelian varieties.

\begin{fact}[Zilber Trichotomy in $\DCF_0$]
\label{dcftrichotomy}
Suppose $p\in S(A)$ is a minimal type.
\begin{itemize}
\item[(a)]
If $p$ is nontrivial locally modular then $p$ is nonorthogonal to the generic type of the Manin kernel of a simple abelian variety over $\acl(A)$ that does not descend to the constants.
\item[(b)]
If $p$ is not locally modular then it is nonorthogonal to the generic type of the field of constants over the empty set.
\end{itemize}
\end{fact}

A few words justifying our formulation of part~(a) may be in order.
It is maybe not widely known that the simple abelian variety can be taken to be over $\acl(A)$.
However, as Martin Hils has pointed out to us, this can be deduced from the results in \cite{hrushovski-sokolovic}.
We give a few details.
By \cite[Lemma 2.11]{hrushovski-sokolovic}, if $p$ is nontrivial locally modular then $p$ is nonorthogonal to the Manin kernel of some simple abelian variety $G_1$ over $Aa_1$ for some $a_1$.
Now take a conjugate $a_2$ of $a_1$ that is independent of $a_1$ over~$A$.
We then get a Manin kernel of some simple abelian variety $G_2$ over $Aa_2$ such that $p$ is also nonorthogonal to its generic type.
Hence, the two Manin kernels are nonorthogonal.
By \cite[Theorem 2.12]{hrushovski-sokolovic}, $G_1$ and $G_2$ are isogenous. By the Claim in the proof of \cite[Proposition 2.8]{hrushovski-sokolovic}, it follows that $G_1$ is isogenous to a (necessarily simple) abelian variety $G$ over $\acl(A)$.
The Manin kernel of $G$ is thus nonorthogonal to that of $G_1$ (by Theorem 2.12 of \cite{hrushovski-sokolovic} again).
So $p$ is nonorthogonal to the Manin kernel of $G$, which is over $\acl(A)$, as desired.

As we have pointed out earlier, a consequence of Fact~\ref{dcftrichotomy} is that Assumption~\ref{assumption} is satisfied in $\DCF_0$, and hence Theorem~\ref{the1} applies.

\medskip
\subsection{Two Examples}
We begin with a pair of examples that show Theorem~\ref{the1} to be best possible in the sense that it is essential in~($\dagger$) to consider arbitrary $b\in\acl(Aa)$.
That is, as Example~\ref{maninkernel} shows, one cannot deduce isolation of $p=\tp(a/A)$  by checking that $\displaystyle a\ind_A q(\U)$ for all nonisolated minimal types $q\in S(A)$.
In fact, as Example~\ref{j} shows, it does not even suffice to consider $q\in S(Ab)$ for all $b\in\dcl(Aa)$, one must pass to $\acl(Aa)$.

\begin{example}[Parametrised family of Manin kernels]
\label{maninkernel}
From~\cite[$\S$4]{pdme} one sees that there exist nonisolated finite rank types $p=\tp(a)$ in $\DCF_0$ satisfying:
\begin{itemize}
\item
$p$ is weakly orthogonal to the field of constants $\mathcal C$; and,
\item
there exists $b\in\dcl(a)$ such that $\tp(b)$ is minimal $\mathcal C$-internal and $\tp(a/b)$ is minimal nontrivial locally modular.
\end{itemize}
We claim that $\displaystyle a\ind q(\U)$ for any $q\in S(\emptyset)$ nonisolated and minimal.
\end{example}

\begin{proof}
Indeed, if $q$ is trivial then it is orthogonal to both $\tp(b)$ and $\tp(a/b)$, and hence to $p$, so that $\displaystyle a\ind q(\U)$ follows.
If $q$ is non-locally-modular then it is the generic type of some strongly minimal $0$-definable set $X$ (see~\cite[$\S2.3$]{gst}).
By nonisolation of~$q$ we must have that $X\cap\acl(\emptyset)$ is infinite.
But then $X(\C)$ is infinite, and hence cofinite.
So $q(\U)\subseteq \C^n$ and $\displaystyle a\ind q(\U)$ follows by weak orthogonality.
As there are no  minimal nontrivial locally modular types over the empty set -- this follows from Fact~\ref{dcftrichotomy}(a) -- these are all the possibilities for $q$.
\end{proof}

The above example works not only over the empty set, but over any subset of the field of constants.
Indeed, for any $A\subseteq \mathcal C$, letting $a$ be as in the above example, one can verify that $\tp(a/A)$ remains a nonisolated type of finite rank while $\displaystyle a\ind_A q(\U)$ for all $q\in S(A)$ nonisolated and minimal.

\begin{example}[The symmetric power of the $j$-function equation]
\label{j}
Freitag and Scanlon~\cite{freitag-scanlon} have studied the order three algebraic differential equation satisfied by the analytic $j$-function.
It defines, over the empty set, a strongly minimal trivial set $X$ in $\DCF_0$.
But unlike all previous such examples, $X$ is not $\omega$-categorical.
Indeed, $X\cap\acl(c)$ is infinite for any $c\in X$ generic.
This is due to Hecke correspondences, see the final paragraph of the proof of~\cite[Theorem~4.7]{freitag-scanlon}.
Now, let $c_1,c_2\in X$ be a pair of independent generics.
Let $a$ be a code for $\{c_1,c_2\}$, and set $p:=\tp(a)$.
We claim that $p$ is nonisolated but $\displaystyle a\ind_{b}q(\U)$ for any $b\in\dcl(a)$ and any nonisolated minimal $q\in S(b)$.
\end{example}

\begin{proof}
Note, first of all, that $\tp(c_1/c_2)$ is not isolated since it is the generic type of $X$ over $c_2$, and $X\cap\acl(c_2)$ is infinite by construction.
In particular, as $a$ is interalgebraic with $(c_1,c_2)$, we have that $p$ is nonisolated.

Now, because $p$ is the code of a set of independent generic elements in a trivial strongly minimal set, we have by~\cite[Example~2.2]{moosapillay14} that $p$ admits no {\em proper fibrations}.
That is, if $b\in\dcl(a)$ then either $a\in\acl(b)$ or $b\in\acl(\emptyset)$.
So it suffices to show that $\displaystyle a\ind q(\U)$ for any nonisolated minimal type $q$ over $\acl(\emptyset)$.
Suppose this fails for some $q$.
Then $\displaystyle (c_1,c_2)\nind q(\U)$.
Since $q$ is minimal, this implies that $q$ is nonorthogonal to the generic type of $X$.
In particular, $\operatorname{RM}(q)=1$.
(Indeed, after taking a nonforking extension, a realisation of $q$ is interalgebraic with a generic element of the strongly minimal set $X$.)
It follows that $q$ is the generic type of some strongly minimal definable set $Y$ over $\acl(\emptyset)$.
Moreover, $Y$ must be trivial since $X$ is.
But then $Y(\mathcal C)$ is finite, so that $Y\cap\acl(\emptyset)$ is finite as $\acl(\emptyset)\subseteq\C$, contradicting the fact that $q$ is nonisolated.
\end{proof}

\medskip
\subsection{An improvement on the main theorem}
The characterisation of isolation given by Theorem~\ref{the1} can be significantly improved in the case of $\DCF_0$.
We begin by pointing out that the trichotomy expressed by Fact~\ref{dcftrichotomy} takes on an even stronger form when we restrict our attention to the nontrivial minimal types that are nonisolated:

\begin{lemma}\label{description}
A nontrivial minimal type $p\in S(A)$ is nonisolated if and only if $p(\U)\subseteq \acl(Aq(\U))$ where $q\in S\big(\acl(A)\big)$ is the generic type of either the constant field or of the Manin kernel of some simple abelian variety over $\acl(A)$ that does not descend to the constants.
\end{lemma}

\begin{proof}
Fact~\ref{dcftrichotomy} tells us that nontriviality implies nonorthogonal to the generic type, say $q\in S\big(\acl(A)\big)$, of either the constant field or of the Manin kernel of some simple abelian variety over $\acl(A)$ that does not descend to the constants.
Suppose $p(\U)\not\subseteq \acl(Aq(\U))$.
By minimality of $p$, this means that there is $a\models p$ such that $\displaystyle a\ind_A q(\U)$.
On the other hand, $p$ is almost internal to $q$.
These conditions, namely $\displaystyle a\ind_A q(\U)$ and the almost internality of $\tp(a/A)$ to $q$, imply that $p$ is isolated (see the argument in the last paragraph of the proof of Theorem~\ref{the1}).

For the converse, note that the generic type of the constant field $\C$ is minimal and nonisolated since $\C$ is a strongly minimal set with infinitely many points in $\acl(A)$, that infinite set being the characteristic zero field $\C\cap\acl(A)$.
The generic type of the Manin kernel of a simple abelian variety over $\acl(A)$ is minimal and nonisolated for the same reason -- the Manin kernel is strongly minimal and it has infinitely many $\acl(A)$-points, namely the torsion points of the abelian variety.
This gives the right-to-left direction, using for example~($\dagger$) applied with $b=\emptyset$, which holds of any isolated type $p$ by Proposition~\ref{isotodagger}.
\end{proof}

We obtain the following improvement of Theorem~\ref{the1} in the case of $\DCF_0$.

\begin{theorem}
\label{the1dcf}
Suppose $p=\tp(a/A)$ is of finite rank.
Then $p$ is isolated if and only if the following hold:
\begin{itemize}
\item[(i)]
$\displaystyle a\ind_A\C$; and
\item[(ii)]
for every $b\in\acl(Aa)$ and $G$ a simple abelian variety over $\acl(Ab)$ that does not descend to the constants, letting $G^\sharp$ denote the Manin kernel of $G$, $\displaystyle a\ind_{Ab}G^\sharp$; and
\item[(iii)]
for every $b\in \acl(Aa)$ and $q\in S(Ab)$ nonisolated and trivial minimal, $\displaystyle a\ind_{Ab}q(\U)$.
\end{itemize}
\end{theorem}

\begin{proof}
Suppose $p$ is isolated.
By Proposition~\ref{isotodagger} we know that~($\dagger$) holds.
So, for every $b\in \acl(Aa)$ and $q\in S(Ab)$ nonisolated and minimal, $\displaystyle a\ind_{Ab}q(\U)$.
Taking $q$ to be nonisolated minimal trivial yields~(iii).
Taking $q$ to be the generic type of $G^\sharp$ over $\acl(Ab)$ we have that $q(\U)=G^\sharp\setminus\acl(Ab)$ and hence $\displaystyle a\ind_{Ab}G^\sharp$.
Similarly, taking $b=\emptyset$ and $q$ to be the generic type of $\C$ over $A$, we have that $q(\U)=\C\setminus\acl(A)$ and hence $\displaystyle a\ind_A\C$.

For the converse, by Theorem~\ref{the1}, it suffices to show that~(i) through~(iii) imply~($\dagger$).
Let $b\in \acl(Aa)$ and $q\in S(Ab)$ nonisolated and minimal.
We want to show $\displaystyle a\ind_{Ab}q(\U)$.
If $q$ is trivial then this follows by~(iii).
If $q$ is nontrivial then by Lemma~\ref{description} we have that
$q(\U)\subseteq \acl(Abq'(\U))$ where $q'\in S(Ab)$ is the generic type of either the constant field or of the Manin kernel of a simple abelian variety over $\acl(Ab)$ that does not descend to the constants.
So it suffices to show that $\displaystyle a\ind_{Ab}q'(\U)$.
If $q'$ is the generic type of a Manin kernel then this is~(ii).
So suppose $q'$ is the generic type of $\C$ over $Ab$.
For any finite tuple $c$ from $q'(\U)\subseteq\C^n$ we have $\displaystyle a\ind_Ac$ by~(i), and so $\displaystyle a\ind_{Ab}c$ as $b\in\acl(Aa)$.
We have shown that $\displaystyle a\ind_{Ab}q'(\U)$, as desired.
\end{proof}

In practice, conditions~(i) and~(ii) of Theorem~\ref{the1dcf} are relatively easy to check as they refer to concrete differential varieties.
It is the trivial case, namely condition~(iii), that remains in general intractable.
Unfortunately, we cannot eliminate this condition, even when $A=\emptyset$.
For example\footnote{We have to resort here to this relatively recently discovered example because all previously known trivial strongly minimal sets in $\DCF_0$ were $\omega$-categorical, and it is easy to see that the generic type $q$ of an $\omega$-categorical strongly minimal set over a differential field that is finitely generated over its constant subfield is always isolated, and hence cannot pose an obstacle to~(iii).}, let $X$ be the trivial strongly minimal but not $\omega$-categorical $0$-definable set coming from~\cite{freitag-scanlon} and discussed in Example~\ref{j} above.
Let $(c_1,c_2) $ be an independent pair of generic elements of $X$, and this time let $a:=(c_1,c_2)$ and $p:=\tp(a)$.
Then $p$ satisfies conditions~(i) and~(ii) of Theorem~\ref{the1dcf}, but fails condition~(iii) with $b=c_2$.
Indeed, $q:=\tp(c_1/c_2)$ is minimal and trivial as it is the generic type of $X$, it is nonisolated since $X$ has infinitely many points algebraic over $c_2$, and clearly $\displaystyle a\nind_{Ac_2}c_1$.

Nonetheless, in some cases we can ignore condition~(iii); for example when $p$ extends a finite rank definable group.
In that case, any trivial minimal type is orthogonal to $p$ and all its extensions, so~(iii) is automatic.
In other cases we can eliminate~(ii) as well; if $p$ is analysable in the constants then it and all its extensions will be orthogonal to all Manin kernels as well as all trivial minimal types.

\medskip
\subsection{Connection to the Dixmier-Moeglin equivalence}
\label{sectdme}

\noindent
In~\cite[$\S$9]{pdme} it was shown that finite rank types $p=\tp(a/A)$ in $\DCF_0$, with $A\subseteq\C$, that satisfy condition~(i) of Theorem~\ref{the1dcf} but are not isolated, can be used to produce finite Gelfand-Kirillov dimension counterexamples to the classical Dixmier-Moeglin equivalence for noetherian algebras.
The types with these properties that arose there were the ones coming from the parameterised Manin kernels of Example~\ref{maninkernel}.
As we have just seen, the generic type of $X\times X$, where $X$ is defined by the differential equation satisfied by the $j$-function, gives us another such counterexample, different from the ones appearing in~\cite{pdme}.

The {\em differential} Dixmier-Moeglin equivalence for $D$-varieties was made explicit in~\cite{BLSM}, where positive results, as opposed to counterexamples, were the focus.
The following is Corollary~2.4 of~\cite{BLSM}.
It was used there to show that $D$-groups over the constants satisfy the differential Dixmier-Moeglin equivalence, and eventually to verify the classical Dixmier-Moeglin equivalence for Hopf Ore extensions.
We give here an alternative proof; indeed it is an immediate consequence of Theorem~\ref{the1dcf} above.
We use freely the terminology of $D$-varieties, and the specific notions developed in~\cite{BLSM}, without further explanation.

\begin{corollary}
\label{deduce-diffDME}
Suppose $(V,s)$ is a $D$-variety over an algebraically closed $\delta$-subfield $A$ of the field of constants, with the property that every irreducible $D$-subvariety of $V$ over~$A$ is compound isotrivial.
Then $(V,s)$ satisfies the differential Dixmier-Moeglin equivalence.
\end{corollary}

\begin{proof}
Working over the constants the Dixmier-Moeglin equivalence reduces to showing that every type in $S(A)$ extending $(V,s)^\sharp$, that is weakly orthogonal to the constants, is isolated.
The compound isotriviality assumption means that every such type $p=\tp(a/A)$ is analysable in the constants.
Conditions~(ii) and~(iii) of Theorem~\ref{the1dcf} are therefore automatically satisfied.
Hence, by Theorem~\ref{the1dcf}, weak orthogonality to $\C$, which is condition~(i), implies isolation.
\end{proof}

\bigskip
\section{The case of $\ccm$}
\label{sectccm}

\noindent
As we have mentioned, Theorem~\ref{the1} also applies to the theory of compact complex manifolds.
Much of what was done in the previous section for $\DCF_0$ goes through for $\ccm$ if one replaces $\C$ by the projective line and Manin kernels by {\em nonstandard simple complex tori} (see~\cite{ams}) of dimension greater than $1$.
This is especially the case if you restrict attention to compact K\"ahler manifolds where one has essential saturation (see~\cite{es}).
There is even an analogue of the algebraic differential equation satisfied by the analytic $j$-function: in~\cite{categorical} it is observed that there exists a $0$-definable strongly minimal trivial set in $\ccm$ that is not $\omega$-categorical.
However, there is one key obstacle to obtaining a full analogue of Theorem~\ref{the1dcf}.
We do not know if the analogue of the Claim in the proof of \cite[Proposition 2.8]{hrushovski-sokolovic} holds: 

\begin{question}
Suppose $G_1$ and $G_2$ are nonstandard simple complex tori over $b_1$ and $b_2$ respectivey, and $b_1$ and $b_2$ are independent over $c\in\dcl(b_1)\cap\dcl(b_2)$.
Is it the case that if $G_1$ and $G_2$ are isogenous then there is $c'\in\acl(c)$ and a nonstandard simple complex torus over $c'$ to which both $G_1$ and $G_2$ are isogenous?
\end{question}

In the algebraic case one uses the fact that an algebraically closed set is a model, which is no longer true in $\ccm$.
Without a positive answer to this question one obtains only a weak analogue of Theorem~\ref{the1dcf} where in condition~(ii) Manin kernels are replaced by nonisolated minimal types that are nonorthogonal to the generic type of a nonstandard simple complex torus defined possibly over extra parameters.

\bigskip


\begin{thebibliography}{10}

\bibitem{ams}
Matthias Aschenbrenner, Rahim Moosa, and Thomas Scanlon.
\newblock Strongly minimal groups in the theory of compact complex spaces.
\newblock {\em J. Symbolic Logic}, 71(2):529--552, 2006.

\bibitem{pdme}
Jason Bell, St\'ephane Launois, Omar~Le\'on S\'anchez, and Rahim Moosa.
\newblock Poisson algebras via model theory and differential-algebraic
  geometry.
\newblock {\em J. Eur. Math. Soc.}, 19(7):2019--2049, 2017.

\bibitem{BLSM}
Jason Bell, Omar Le\'on~S\'anchez, and Rahim Moosa.
\newblock ${D}$-groups and the {D}ixmier-{M}oeglin equivalence.
\newblock {\em {A}lgebra \& {N}umber {T}heory}, 12(2):343--378, 2018.

\bibitem{zoe-udi}
Zo\'e Chatzidakis and Ehud Hrushovski.
\newblock Model theory of difference fields.
\newblock {\em Transactions of the American Mathematical Society}, 351(8):2997
  -- 3071, 1999.

\bibitem{freitag-scanlon}
James Freitag and Thomas Scanlon.
\newblock Strong minimality and the $j$-function.
\newblock To appear in {\em J. Eur. Math. Soc.}

\bibitem{udi-locmod}
Ehud Hrushovski.
\newblock Locally modular regular types.
\newblock In {\em Classification theory ({C}hicago, {IL}, 1985)}, volume 1292
  of {\em Lecture Notes in Math.}, pages 132--164. Springer, Berlin, 1987.

\bibitem{hrushovski-sokolovic}
Ehud Hrushovski and Zeljko Sokolovic.
\newblock Minimal sets in differentially closed fields.
\newblock Unpublished.

\bibitem{markermanin}
David Marker.
\newblock Manin kernels.
\newblock In {\em Connections between model theory and algebraic and analytic
  geometry}, volume~6 of {\em Quad. Mat.}, pages 1--21. Dept. Math., Seconda
  Univ. Napoli, Caserta. Available at
  \url{http://homepages.math.uic.edu/~marker/manin.ps}, 2000.

\bibitem{mof}
David Marker, Margit Messmer, and Anand Pillay.
\newblock {\em Model theory of fields}, volume~5 of {\em Lecture Notes in
  Logic}.
\newblock Association for Symbolic Logic, La Jolla, CA; A K Peters, Ltd.,
  Wellesley, MA, second edition, 2006.

\bibitem{es}
Rahim Moosa.
\newblock On saturation and the model theory of compact {K}\"ahler manifolds.
\newblock {\em J. Reine Angew. Math.}, 586:1--20, 2005.

\bibitem{categorical}
Rahim Moosa and Anand Pillay.
\newblock {$\aleph_0$}-categorical strongly minimal compact complex manifolds.
\newblock {\em Proc. Amer. Math. Soc.}, 140(5):1785--1801, 2012.

\bibitem{moosapillay14}
Rahim Moosa and Anand Pillay.
\newblock Some model theory of fibrations and algebraic reductions.
\newblock {\em Selecta Math. (N.S.)}, 20(4):1067--1082, 2014.

\bibitem{gst}
Anand Pillay.
\newblock {\em Geometric stability theory}, volume~32 of {\em Oxford Logic
  Guides}.
\newblock The Clarendon Press, Oxford University Press, New York, 1996.
\newblock Oxford Science Publications.

\bibitem{pillayziegler}
Anand Pillay and Martin Ziegler.
\newblock Jet spaces of varieties over differential and difference fields.
\newblock {\em Selecta Math. (N.S.)}, 9(4):579--599, 2003.

\end{thebibliography}

\end{document}